\documentclass[bibalpha]{amsart}
\usepackage{verbatim}
\usepackage{enumerate}

\title{NIP henselian valued fields}

\author{Franziska Jahnke} 
\address{Institut f\"ur Mathematische Logik und Grundlagenforschung\\Einsteinstr. 62\\48149 M\"unster, 
Germany}
\email{franziska.jahnke@uni-muenster.de}

\author{Pierre Simon}
\address{Department of Mathematics\\University of California at 
Berkeley\\ 
733 Evans Hall\\Berkeley, CA 94720-3840, USA}
\email{simon@math.berkeley.edu }
\thanks{Partially supported by the Deutsche Forschungsgemeinschaft (DFG, German Research Foundation) via CRC 878 and 
under Germany's Excellence Strategy
EXC 2044--390685587, `Mathematics M\"unster: Dynamics--Geometry--Structure', ValCoMo (ANR-13-BS01-0006), NSF (grant no. 1665491) and a Sloan fellowship.}

\usepackage{enumerate}
\usepackage{braket}
\usepackage{txfonts}
\newtheorem{Th}{Theorem}[section]

\newtheorem*{Thm*}{Theorem}

\newtheorem*{Def}{Definition}
\newtheorem{Cor}[Th]{Corollary}
\newtheorem{Prop}[Th]{Proposition}

\newtheorem*{Rem}{Remark}

\newtheorem{Lem}[Th]{Lemma}

\newtheorem{Thm}[Th]{Theorem}

\newtheorem*{Lem*}{Lemma}
\numberwithin{equation}{section}

\usepackage{amsmath}
\usepackage{amsfonts}
\usepackage{extarrows}
\usepackage{amsthm}
\usepackage{amssymb}
\usepackage{bm}
\usepackage[applemac]{inputenc}
\usepackage{enumerate}
\newtheoremstyle{mystyle}{}{}{\slshape}{2pt}{\scshape}{.}{ }{} 
\newtheorem{thm}{Theorem}[section]
\newtheorem{cor}[thm]{Corollary}
\newtheorem{prop}[thm]{Proposition}
\newtheorem{lemme}[thm]{Lemma}

\newtheorem{fact}[thm]{Fact}

\theoremstyle{definition}

\theoremstyle{mystyle}

\theoremstyle{remark}

\newcommand{\monster}{\mathcal U}

\DeclareMathOperator{\tp}{tp}

\def\indsym#1#2{%
 \setbox0=\hbox{$\m@th#1x$}%
 \kern\wd0%
 \hbox to 0pt{\hss$\m@th#1\mid$\hbox to 0pt{$\m@th#1^{#2}$\hss}\hss}%
 \lower.9\ht0\hbox to 0pt{\hss$\m@th#1\smile$\hss}%
 \kern\wd0}

\def\nindsym#1#2{%
 \setbox0=\hbox{$\m@th#1x$}%
 \kern\wd0%
 \hbox to 0pt{\hss$\m@th#1\not$\kern1.4\wd0\hss}
 \hbox to 0pt{\hss$\m@th#1\mid$\hbox to 0pt{$\m@th#1^{#2}$\hss}\hss}%
 \lower.9\ht0\hbox to 0pt{\hss$\m@th#1\smile$\hss}%
 \kern\wd0}

\begin{document}
\begin{abstract} 
We show that any theory of tame henselian valued fields is NIP if and only if 
the theory of its residue field and the theory of its value group are NIP.
Moreover, we show that if $(K,v)$ is a henselian valued field of 
residue characteristic $\mathrm{char}(Kv)=p$ such that if $p>0$, 
depending on the characteristic of $K$ either the degree of imperfection or the 
index of the $p$th powers is finite, 
then $(K,v)$ is NIP iff $Kv$ is NIP and $v$ is roughly separably tame.
\end{abstract}

\maketitle

\section{Introduction}
In this paper, we study NIP henselian valued fields. More precisely, we consider the question of when NIP transfers from 
the residue field to the valued field.
Our approach generalizes well-known results of Delon and Gurevich-Schmitt 
for henselian fields of equicharacteristic $0$, and of B\'elair for
certain
perfect henselian fields of positive characteristic, in a uniform way. 

\begin{fact}[Delon-Gurevich-Schmitt]
Let $(K,v)$ be a henselian valued field of residue characteristic $\mathrm{char}(Kv)=0$. Then
$$(K,v) \textrm{ is NIP as a valued field }\Longleftrightarrow Kv \textrm{ is NIP as a pure field.}$$
\end{fact}
Delon first proved this theorem with the additional condition that the 
value group $vK$ is also NIP as an ordered abelian group
(\cite{Del}). Gurevich and Schmitt showed that the theory of any ordered abelian group is NIP (\cite[Theorem 3.1]{GS84}).

B\'elair has provided a positive characteristic analogue of this theorem, namely
\begin{fact}[{\cite[Corollaire 7.5]{Bel}}]
Let $(K,v)$ be an algebraically maximal Kaplansky field of characteristic $\mathrm{char}(K)=p$. Then
$$(K,v) \textrm{ is NIP as a valued field }\Longleftrightarrow Kv \textrm{ is NIP as a pure field.}$$
\end{fact}
See section \ref{trans} for the definition of an algebraically maximal Kaplansky field.
The first result of this paper is a generalization of these two results. 
Our main ingredients are the
algebra and model theory of separably tame
valued fields, in particular the Ax-Kochen Ershov Theorem for separably tame valued fields of a fixed finite degree of
imperfection, as developed by
Kuhlmann (\cite{Kuh13}) and by Kuhlmann and Pal (\cite{Kuh14}). The definition of separably tame is also given in section \ref{trans}.
Our first main result is the following Theorem, which we prove as Theorem \ref{amK} in section \ref{trans}:
\begin{Thm*}
Any complete theory of separably algebraically maximal 
Kaplansky fields of finite degree of imperfection 
is NIP if the corresponding theories of residue fields and value groups are both NIP.
\end{Thm*}
In particular, our result allows additional NIP structure on 
the value group and 
the residue field (a fact which is not 
explicit in either Delon's or B\'elair's work). 

We also ask to what extent an NIP henselian valued field is already 
separably tame, or can be decomposed into separably tame
parts. Here, our main ingredients are Johnson's ideas from \cite{Joh15} and the results obtained by Kaplan, Scanlon and Wagner
in \cite{KSW}, in particular the fact that an NIP field of positive characteristic is either finite or Artin-Schreier closed.
Our second main result (proven as Theorem \ref{main2} in section \ref{sec4}) is the following
\begin{Thm*} 
Let $(K,v)$ be a henselian valued field of residue characteristic 
$\mathrm{char}(Kv)=p$. In case $(K,v)$ has mixed characteristic, 
assume that
$K^\times/(K^\times)^p$ is finite. 
If the characteristic of $K$ is positive, assume that $K$ has finite degree
of imperfection.
Then
$$(K,v) \textrm{ is NIP }\Longleftrightarrow Kv \textrm{ is NIP and }(K,v)\textrm{ is roughly separably tame}.$$ 
\end{Thm*}

The paper is organized as follows. In section \ref{crit}, we define two properties of a valued field, namely (SE) and (Im). (SE) ensures
that the residue field and value group are stably embedded, (Im) makes sure that the type of an element generating an immediate extension is implied
by NIP formulae. The main result in this section is Theorem \ref{SEIm}: 
Assuming (SE) and (Im), we get an NIP transfer 
from the theories of value group and residue field to the theory of the valued field. The last result of the section is Proposition \ref{se}
which proves that when we have an NIP structure and a stably embedded definable set which carries some additional NIP
structure, then the expansion by this extra structure remains NIP.

In section \ref{trans}, we show that the assumptions of Theorem \ref{SEIm} are satisfied in separably algebraically maximal
Kaplansky fields of a fixed finite degree of imperfection (Theorem \ref{amK}). This gives us the desired NIP transfer theorem
for these fields.

Finally, in section \ref{sec4}, we use the techniques from \cite{Joh15} to show Theorem \ref{main2}.

\section{An NIP transfer principle} \label{crit}

Recall that a definable set $D$ is said to be \emph{stably embedded} 
if for every formula $\phi(x;y)$, $y$ a finite tuple of variables from the same sort as $D$, there is a formula $d\phi(z;y)$ such that for any $a\in \monster^{|x|}$, there is a tuple $b\in D^{|z|}$, such that $\phi(a;D)=d\phi(b;D)$.

\medskip

Let $T$ be a complete theory of valued fields with possible additional structure. We assume the following:

\begin{description}
\item[(SE)] The residue field and the value group are stably embedded.

\item[(Im)] If $K\models T$ and $a\in \monster$ is a singleton such that $K(a)/K$ is an immediate extension, then $\tp(a/K)$ is implied by instances of NIP formulas.
\end{description}

\begin{lemme}
Let $I$ be indiscernible over a set $A$ and let $\bar d \in D$, where $D$ is $\emptyset$-definable, stably embedded with NIP induced structure. Then no formula with parameters in $A\bar d$ can have infinite alternation on $I$. 
\end{lemme}
\begin{proof}
Let $\phi(x;\bar d)$ be such a formula, where we hide the parameters from $A$. Write $I=(a_i:i\in \mathcal I)$. As $D$ is stably embedded, there is a formula $\psi(z;\bar y)$ and for each $i$ a parameter $c_i$ in $D$ such that $\phi(a_i;D)=\psi(c_i;D)$. If the formula $\phi(x;\bar d)$ alternates infinitely often on $I$, then by indiscernibility of $I$, for every subset $I_0$ of $I$, 
we can find some $\bar d_{I_0}\in D$ such that $\phi(I;\bar d_{I_0})=I_0$. But then the same can be done for $J=(c_i:i<\omega)$, contradicting NIP on $D$.
\end{proof}

\begin{lemme}\label{lem_techNIP}
Assume that $(a_i:i<\omega)$ is an indiscernible sequence. For every $i$, let $b_{2i}$ be a tuple coming from a stably embedded sort $D$ on which the induced structure is NIP. Assume that the sequences $(a_{2i}b_{2i}:i<\omega)$ and $(a_{2i}a_{2i+1}:i<\omega)$ are indiscernible over some tuple $a$. Then we can find tuples $b_{2i+1}$ and $a'_i$ such that:

$\bullet_1$ $a'_{2i}=a_{2i}$,

$\bullet_2$ $\tp((a'_i:i<\omega)/a)=\tp((a_i:i<\omega)/a)$,

$\bullet_3$ the sequence $(a'_{i}b_{i}:i<\omega)$ is indiscernible.
\end{lemme}
\begin{proof}
First, extend the sequence $(a_i:i<\omega)$ to a very long one $(a_i:i<\kappa)$ with the same properties. Pick an increasing sequence $\lambda_0<\lambda_1<\cdots$ of even elements of $\kappa$, far apart from each other. Set $C= (a_{\lambda_i} b_{\lambda_i}:i<\omega)$.

By the previous lemma, for each $k<\omega$, we can find $\lambda_k < i_k < \lambda_{k+1}$ such that $i_k$ is even and the two sequences $(a_{i_k}:k<\omega)$ and $(a_{i_k+1}:k<\omega)$ have the same type over $C$. Hence we can find points $b'_{k}$ such that the two sequences $(a_{i_k}b_{i_k}:k<\omega)$ and $(a_{i_k+1}b'_k:k<\omega)$ have the same type over $C$. Now, let $\sigma$ be an automorphism over $a$ which sends $C$ to $(a_{2i}b_{2i})_{i<\omega}$, and set $a'_{2k+1} = \sigma(a_{i_k+1})$ and $b_{2k+1} = \sigma(b'_k)$.
\end{proof}

\begin{thm} \label{SEIm}
Under assumptions (SE) and (Im), the theory $T$ is NIP iff the theories of the residue field and value group are.
\end{thm}
\begin{proof}
Assume that $T$ is not NIP. Then we can find an indiscernible sequence $(a_i:i<\omega)$, a singleton $a$ and a formula $\phi(x;y)$ such that $\phi(a;a_i)$ holds if and only if $i$ is even. By Ramsey and compactness, we may assume that the sequence of pairs $(a_{2i}a_{2i+1}:i<\omega)$ is indiscernible over $a$. We may also increase each $a_i$ so that it enumerates a model $M_i$. Now, for each $i$, let $b_{2i}$ be an enumeration of the residue field of $M_{2i}(a)$ and $c_{2i}$ an enumeration of its value group. Applying Lemma \ref{lem_techNIP} twice: for the value group and for the residue field and changing the points $a_{2i+1}$, we can find tuples $b_{2i+1}$ and $c_{2i+1}$ such that $(a_i\hat{~}b_i\hat{~}c_i)_{i<\omega}$ is indiscernible and $\tp((a_i)_{i<\omega}/a)$ is preserved.
Then, we can extend every tuple $a_i\hat{~}b_i\hat{~}c_i$ to a model. Iterate this $\omega$ times so as to have the following:

$\boxtimes$ an indiscernible sequence $(N_i:i<\omega)$ of models such that the even places extend the original $a_i$, the type of the sequence over $a$ extends the initial one, $(N_{2i}N_{2i+1}:i<\omega)$ is indiscernible over $a$ and $N_0(a)/N_0$ is immediate.

Now by assumption (Im), the type $\tp(a/N_0)$ is implied by NIP formulas. Any such formula can only alternate finitely often on the sequence $(N_i:i<\omega)$. 
Hence we must have $\tp(N_0,a)=\tp(N_1,a)$, 
contradicting the initial assumption.
\end{proof}

\begin{Rem} \label{NTP}
One can can also prove Theorem \ref{SEIm} in the NTP$_2$ 
context, i.e. one can show that under assumptions (SE) and (Im), the 
theory $T$ is NTP$_2$ if and only if the theories of the residue field and value group are.
We do not include a proof here as the result is implicit in \cite{CH14},
using \cite[Lemma 3.8]{CH14} instead of Lemma \ref{lem_techNIP}.
\end{Rem}

\begin{prop} \label{se}
Let $T$ be NIP in a relational language $L$, let $M\models T$ and let $D$ be a definable set. Assume that $D$ is stably embedded. Let $D_{ind}$ be the structure with universe $D(M)$ and the induced $L$-structure. Consider an expansion $D_{ind}\subseteq D'$ to a relational language $L_p$ and let $M'$ be the corresponding expansion of $M$ in the language $L'=L\cup L_p$. Then the definable set $D$ is stably embedded in $M'$. Furthermore, if $D'$ is NIP, then so is $M'$.
\end{prop}
\begin{proof}
This is more or less implicit in \cite{ExtDef2}. We may assume that $D'$ admits elimination of quantifiers in the relational language $L_p$ and also that $M$ admits elimination of quantifiers in $L$. Call a formula $D$-bounded if it is of the form $Q_1 z_1 \in D \ldots Q_n z_n \in D \bigvee_{i<m} \phi_i(\bar x,\bar z)\wedge \chi_i(\bar x,\bar z)$, where $\phi_i(\bar x,\bar z)$ is a quantifier-free $L$-formula and $\chi_i(\bar x,\bar z)$ is a quantifier-free $L_p$-formula (with all variables restricted to $D$). Lemma 46 in \cite{ExtDef2} implies that every $L'$-formula is equivalent to a $D$-bounded formula. (This lemma is stated and proven in the case where the expansion is by adding a unique unary predicate, but works just as well in the general case.) This implies that $D$ remains stably embedded in the expansion to $L'$ and the induced structure is exactly the one coming from $L_p$.

It remains to show that every $D$-bounded formula is NIP in $M'$. This is proved exactly as Theorem 2.4 in \cite{DepPairs} (except that we do not need honest definitions since we assume stable embeddedness). We give details. Let $\phi(x;y)$ be a $D$-bounded formula. We induct on the number of quantifiers $(Qz\in D)$ at the beginning of $\phi$. If there are none, then $\phi$ is a boolean combination of $L$ and $L_p$-formulas and hence is NIP. Now assume that $\phi(x;y)=(\exists z\in D)\psi(xz;y)$, where $\psi(xz;y)$ is a $D$-bounded formula, which we can assume to be NIP by induction. Work in some sufficiently saturated model $N$. Assume that $\phi$ has IP and let $(a_i:i<\omega)$ be an indiscernible sequence of tuples of size $|x|$ and $c$ such that $\phi(a_i;c)$ holds if and only if $i$ is even. Hence for even $i$, we can find $b_i \in D$ such that $\psi(a_ib_i;c)$ holds. By Lemma \ref{lem_techNIP} we can find points $b_i$ for odd $i$ such that the full sequence of pairs $(a_ib_i:i<\omega)$ is indiscernible. For $i$ odd, $\neg (\exists z\in D)\psi(a_iz;c)$ holds by hypothesis and in particular $\models \neg \psi(a_ib_i;c)$. Therefore the formula $\psi(xz;c)$ alternates on the sequence $(a_ib_i)_{i<\omega}$ contradicting NIP.
\end{proof}

\section{An NIP transfer principle for tame and some separably tame fields} \label{trans}
In this section, we apply the results from the previous section to get an NIP transfer principle for tame fields.

\begin{Def}
Let $(K,v)$ be a valued field and $p=\mathrm{char}(Kv)$. 
\begin{enumerate}
\item
We say that $(K,v)$ is \emph{(separably) algebraically maximal} if $(K,v)$ has no proper immediate (separable) algebraic extensions.
\item We say that $(K,v)$ is \emph{(separably) tame} if the value group $vK$ is $p$-divisible, the residue field $Kv$ is perfect
and $(K,v)$ is (separably) algebraically maximal.
\end{enumerate}
\end{Def}
 
Note that the definitions of tame and separably tame are not those given in \cite{Kuh13}, but by 
\cite[Theorem 3.2]{Kuh13} (respectively \cite[Theorem 3.10]{Kuh13})
these are equivalent to the ones given there. If $K$ is perfect, then tameness and separable tameness coincide. Both
tameness and separable tameness imply henselianity.

Recall that for a field $K$ of characteristic $p>0$, we have that $K^p$ is a subfield of $K$ and hence $K$ is a vector space
over $K^p$. 
The vector space dimension of $K$ over $K^p$ is always some power $p^e$ of $p$.
The exponent $e$ is called the \emph{degree of imperfection} of $K$. 
A field is perfect
if and only if its degree of imperfection is $0$.

\begin{lemme} \label{SE}
Any complete theory of separably tame fields of finite degree of imperfection 
satisfies condition (SE). In fact, in any separably tame field of finite
degree of imperfection, both 
the residue field and the value group are purely stably embedded.
\end{lemme}
\begin{proof}
Let $(K,v)$ be a separably tame valued field of finite degree of imperfection and take $(K,v)\prec (L,v)$ a sufficiently saturated elementary extension. Let $a,b\in Lv$ having the same type in the pure field $Lv$ over $Kv$. We need to show that $a$ and $b$ have the same type over $K$ in $(L,v)$. Let $L_0$ be an $\aleph_0$-saturated elementary submodel of $L$ containing $Ka$. Note that by elementarity of $(K,v)\prec (L_0,v)$, 
the extension $L_0v|Kv$ is separable and the quotient $vL_0/vK$ is torsion free. Fix some $\sigma: L_0v \to Lv$ which is an elementary embedding over $Kv$ (in the pure field language) so that $\sigma(a)=b$ and $\rho: vL_0 \to vL$ an elementary embedding over $vK$ in the pure group language. By the Separable Relative Embedding Property of \cite[Section 4]{Kuh14}, which holds in separably
tame valued fields by Theorem 5.1 of the
same article, we can find an embedding $\iota : (L_0,v) \to (L,v)$ fixing $K$ and inducing $\sigma$ and $\rho$.

By the same \cite[Theorem 5.1]{Kuh14}, separably tame valued fields of finite degree of imperfection are separably relatively model complete: 
if $(K,v) \subseteq (L,w)$ is a separable extension of separably tame valued 
fields of the same finite degree of imperfection with $Kv \prec Lw$ and 
$vK \prec wL$, one has $(K,v)\prec (L,w)$. This implies that $(\iota (L_0),v)\prec (L,v)$. Therefore we have: \[\tp_L(a/K)=\tp_{L_0}(a/K)=\tp_{\iota(L_0)}(b/K)=\tp_{L}(b/K),\] and thus \[ \tp_{Lv}(a/K)\vdash \tp_{L}(a/K),\]
where the first type is in the pure field $Lv$.

This implies that the residue field is purely stably embedded. A similar argument shows that the value group is purely stably embedded.
\end{proof}

In order to prove (Im) for a suitable class of fields, we need another definition:
\begin{Def}
Let $(K,v)$ be a valued field of residue characteristic $p$. We say that $(K,v)$ is \emph{Kaplansky} if the value group $vK$
is $p$-divisible and the residue field $Kv$ admits no finite extensions of degree divisible by $p$.
\end{Def}
Note that if a valued field is algebraically maximal and Kaplansky, then it is in particular tame. The converse does not hold.

\begin{lemme} \label{Im}
Any complete theory of separably algebraically maximal Kaplansky fields of 
finite degree of imperfection satisfies (Im).
\end{lemme}
\begin{proof}
It is enough to show that if $K(a)/K$ is immediate, 
then $\tp(a/K)$ is given by quantifier-free formulas. 
Quantifier-free formulas describe the isomorphism type of $K(a)$. Let 
$K(\sqrt[p^\infty]{a})$ be the subfield of the perfect hull of $K(a)$ 
containing all $p$-power roots of $a$. Both $K(\sqrt[p^\infty]{a})$ 
and its henselization $K(\sqrt[p^\infty]{a})^h$ are unique up to isomorphism
over $K(a)$ and are separable extensions of $K$. 
Since we are working with Kaplansky fields, 
the tame closure (i.e. the maximal separable algebraic immediate extension) 
$\tilde K$ of $K(\sqrt[p^\infty]{a})^h$ 
is unique up to isomorphism (\cite[Theorem 5.3]{KPR}). 
It is also a separable and immediate extension of $K$. 
By \cite[Theorem 6.2]{Kuh14}, the extension $K\prec \tilde K$ is elementary. 
Therefore the quantifier-free type of 
$a$ over $K$ entirely describes a model containing $a$ and hence implies the full type of $a$ over $K$.
\end{proof}

Combining Lemmas \ref{SE} and \ref{Im} with Theorem \ref{SEIm}, we obtain our first main result:
\begin{thm} \label{amK}
A complete theory of separably algebraically maximal 
Kaplansky fields of finite degree of imperfection 
is NIP if the corresponding theories of residue fields and value groups are both NIP.
\end{thm}

\begin{Rem}
Using Remark \ref{NTP}, Theorem \ref{amK} also holds when we replace NIP by 
NTP$_2$.
\end{Rem}

\begin{cor} \label{equi}
Let $(K,v)$ be a henselian valued field of equicharacteristic $p \geq 0$ 
and of finite degree of imperfection. Assume that $v$ is separably tame and that $Kv$ is infinite and NIP.
Then $(K,v)$ is NIP.
\end{cor}
\begin{proof} In case $p=0$ there is nothing to prove. Assume $p>0$.
The only thing left to show is that $(K,v)$ is Kaplansky, i.e., that $Kv$ has no separable extensions of degree divisible by $p$. As
we have assumed $Kv$ to be infinite NIP, it has no separable extensions of degree divisible by $p$ by \cite[Corollary 4.4]{KSW}. 
\end{proof}

\section{Henselian NIP fields and tameness}
This section is strongly influenced by Will Johnson's results in \cite{Joh15}
and \cite{Joh16}.
\begin{Def}
Let $(K,v)$ be a valued field and $p>0$ a prime. 
\begin{enumerate}
\item We say that $(K,v)$ is 
\emph{roughly $p$-divisible} if $[-v(p),v(p)] \subseteq p\cdot vK$
where $[-v(p),v(p)]$ denotes 
\begin{itemize}
\item $\{0\}$ in case $\mathrm{char}(Kv)\neq p$,
\item $vK$ in case $\mathrm{char}(K)= p$ and
\item the interval $[-v(p),v(p)] \subseteq vK$ in case $\mathrm{char}(K,Kv)=(0,p)$.
\end{itemize}
\item We say that $(K,v)$ is \emph{finitely ramified} if $\mathrm{char}(K,Kv)=(0,p)$ and the interval
$$[-v(p),v(p)] \subseteq vK$$ is finite.
\item Assume $\mathrm{char}(K,Kv)=(0,p)$ and let $\Delta_0$ denote the 
largest convex
subgroup of $vK$ not containing $v(p)$.
We say that $(K,v)$ is \emph{finitely ramified by $p$-divisible} 
if $\Delta_0$ is $p$-divisible and the induced
valuation $\bar{v}:K \twoheadrightarrow vK/\Delta_0 \cup \{\infty\}$ is finitely ramified.
\end{enumerate}
\end{Def}
Note that any finitely ramified $(K,v)$ is in particular finitely ramified
by $p$-divisible since in this case $\Delta_0=\{0\}$ is $p$-divisible.

\begin{Def}
Let $(K,v)$ be a henselian valued field and let $p=\mathrm{char}(Kv)$. 
We say that $(K,v)$ is \emph{roughly (separably) tame} if it satisfies all of the
following properties:
\begin{enumerate}
\item $Kv$ is perfect,
\item $v$ is (separably) algebraically maximal,
\item $(K,v)$ is roughly $p$-divisible or finitely ramified by $p$-divisible,
\item if $Kv$ is finite then $(K,v)$ is finitely ramified.
\end{enumerate}
\end{Def}
Note that in particular every perfect roughly separably tame field is roughly tame and that
every (separably) tame field is roughly (separably) tame.

\begin{Prop}\label{prop_pp}
Let $(K,v)$ be a non-trivially valued NIP field of equicharacteristic $p$. Then $(K,v)$ is Kaplansky.
\end{Prop}
\begin{proof} Any valued field of residue characteristic $\mathrm{char}(Kv)=0$
is Kaplansky. So we assume $p>0$ for the rest of the proof.
As $Kv$ is interpretable in $(K,v)$, $Kv$ is NIP and moreover, 
by \cite[Proposition 5.3]{KSW},
infinite. 
Thus, $Kv$ admits no separable extensions of degree divisible by $p$, by \cite[Corollary 4.4]{KSW}. 
The $p$-divisibility of the value group follows from \cite[Proposition 5.4]{KSW}.
%
%
To see that $Kv$ is perfect, let $\bar a\in Kv$. Pick some $a\in K$ of residue 
$\bar a$ and $c\in \mathfrak{m}_v$. Consider the polynomial $P(X)=X^p + c X - a$. As $K$ has no separable extension of degree divisible by $p$, this polynomial factors in $K$. As the valuation ring $\mathcal O_v$ is integrally closed, there is a factorization $P = Q \cdot R$ with $Q, R \in \mathcal O_v[X]$ non-constant and monic. But then this descends in $Kv$ to a factorization of $\bar P = X^p - \bar a$. This implies that $\bar a$ has a $p$-th root in $Kv$ and hence $Kv$ is perfect.
\end{proof}

\begin{Cor} \label{Cor4.1}
Any NIP henselian valued field $(K,v)$ of equicharacteristic $p$ is separably tame.
\end{Cor}
\begin{proof} The statement is clear for henselian fields of equicharacteristic $0$.
Assume that $(K,v)$ is an NIP henselian valued field and that $\mathrm{char}(K)=p>0$ holds. In particular, $K$
is either finite or has no separable extensions of degree divisible by $p$ by \cite[Corollary 4.4]{KSW}.
If $K$ is finite, $v$ is trivial and $(K,v)$ is separably tame. Assume that $K$ is infinite.
By Proposition \ref{prop_pp}, all that is left to show is that $v$ is separably algebraically maximal. 
This follows from the Lemma of Ostrowski
(see \cite[Theorem 3.3.3]{EP05}):
The degree of any immediate finite Galois extension $(K,v) \subseteq (L,w)$ is divisible by $p$. 
Thus, $(K,v)$ is indeed separably algebraically maximal.
\end{proof}
In \cite[p. 25]{Joh15}, Johnson remarks that his proof of \cite[Lemma 6.8]{Joh15} actually applies to strongly dependent fields. This fact is also proven in
\cite[Lemma 4.1.1 and Theorem 4.3.1]{Joh16}. 
\begin{Prop}[Johnson]
Let $(K,v)$ be a strongly dependent henselian field of mixed characteristic
$(0,p)$.
Then $(K,v)$ is algebraically maximal, $Kv$ is perfect and 
$(K,v)$ is either finitely ramified or roughly $p$-divisible.
\end{Prop}

We now generalize the proof of \cite[Lemma 6.8]{Joh15} slightly to show an NIP version of this statement. As NIP (unlike strong dependence)
does not imply field perfection, we have to add an extra assumption. 
We need the
following lemma, which is very well-known. The second part of the 
statement
is also proven (with essentially the same proof) 
in \cite[Proposition 3.2(a)]{Koe04}. We give the short proof for the 
convenience
of the reader.
\begin{Lem} \label{finite}
Let $(K,v)$ be a valued field. Then, the following holds:
\begin{enumerate}
\item for any $n>0$ we have
$$|K^\times/(K^\times)^n|\geq |Kv^\times/(Kv^\times)^n|,$$ 
\item  in case
$\mathrm{char}(K)\neq \mathrm{char}(Kv)=p>0$ and the index 
$|K^\times/(K^\times)^p|$
is finite
then $Kv$ is perfect.
\end{enumerate}
\end{Lem}
\begin{proof}
\begin{enumerate}
\item
Let $\alpha, \beta \in Kv^\times$ be such that $\alpha \not\equiv \beta 
\,\mathrm{mod}\, (Kv^\times)^n$. 
Then for any lifts $a, b \in \mathcal{O}_v^\times$ of $\alpha, \beta$ one has 
$a \not\equiv b \,\mathrm{mod}\, (K^\times)^n$. Indeed, else there would be
$c \in \mathcal{O}_v^\times$ with $a = c^n b$. But then 
$\alpha = \gamma^n\beta$, for 
$\gamma = res(c) \in Kv^\times$, a contradiction.
\item Assume that $\mathrm{char}(K)\neq \mathrm{char}(Kv)=p>0$ and $Kv$
is imperfect (and hence in particular infinite). 
First, note that $|Kv^\times/(Kv^\times)^p|$ is infinite: since $Kv^p$ is a
proper subfield of $Kv$, the quotient $|Kv^\times/(Kv^p)^\times|$ is
infinite (as it is the projectivization of an $Kv^p$-vector space of dimension
at least $2$). The first part of the lemma now implies 
$|K^\times/(K^\times)^p|$ infinite.
\end{enumerate}
\end{proof}

\begin{Prop}\label{rtame}
Let $(K,v)$ be a henselian valued NIP field of mixed characteristic 
$(0,p)$ such that
$K^\times/(K^\times)^p$ is finite. Then $(K,v)$ is roughly separably tame.
\end{Prop}
\begin{proof}
This proof follows the proof of \cite[Lemma 6.8]{Joh15} closely.
Note that since we have $\mathrm{char}(K)=0$, the notions of rough tameness and 
rough separable tameness coincide.
Moreover, we may assume that $(K,v)$ is saturated as being 
roughly tame is preserved under elementary equivalence of valued fields. 
  
We write $\Gamma:=vK$.
Let $\Delta_0 \leq \Gamma$ be the biggest convex subgroup not containing $v(p)$ and let $\Delta \leq \Gamma$ be
the smallest convex subgroup containing $v(p)$. We get the following decomposition of the place $\varphi_v:K \to Kv$
corresponding to $v$:
$$K =K_0 \xlongrightarrow{\Gamma/\Delta} K_1 \xlongrightarrow{\Delta/\Delta_0} K_2  \xlongrightarrow{\Delta_0} K_3=Kv$$
where every arrow is labelled with the corresponding value group.
Note that $\mathrm{char}(K)=\mathrm{char}(K_1)=0$ and $\mathrm{char}(K_2)= \mathrm{char}(Kv)=p$.
Let $v_i$ denote the valuation on $K_i$ corresponding to the place $K_i \to K_{i+1}$.

\smallskip \noindent
\emph{Claim 1:} The fields $K_i$ are all NIP (as pure fields).

\smallskip
\emph{Proof of Claim 1:} By assumption, $K=K_0$ is NIP and - as $Kv$ is 
interpretable in $(K,v)$ - so is $K_3=Kv$. 
Any convex subgroup of an ordered abelian group is externally definable (i.e.~definable with parameters from the monster
model), in particular both $\Delta$ and $\Delta_0$ are definable in the Shelah expansion $K^{Sh}$ of $K$. As
$K^{Sh}$ also has NIP (\cite[Corollary 3.24]{Simon:book}), all three places are interpretable in an NIP field.
In particular, $K_1$ and $K_2$ are NIP.
\smallskip

\noindent\emph{Claim 2:} If $w$ is a (not necessarily proper) coarsening of $v$ on $K$, then $Kw$ is perfect. 

\smallskip
\emph{Proof of Claim 2:}
For any coarsening $w$ of $v$ with $\mathrm{char}(Kw)=0$ the residue field $Kw$ is perfect. The coarsest coarsening $u$ of
$v$ with $\mathrm{char}(Ku)=p$ corresponds exactly to the place $K \to K_2$. Note that we have 
$\mathcal{O}_{v_1}[\frac{1}{p}]=K_1$ by construction. 
As $|K_1^\times/(K_1^\times)^p|$ is finite by the first part of Lemma 
\ref{finite}, the second part of Lemma \ref{finite} 
implies that $K_2$ is perfect. By Claim 1
and Proposition 
\ref{prop_pp}, all valuation rings $\mathcal{O}_w \subseteq K$ with $\mathcal{O}_v \subseteq\mathcal{O}_w \subseteq\mathcal{O}_u$ have perfect residue field.

\smallskip
In particular, Claim 2 implies that $Kv$ is perfect.

\smallskip
\noindent\emph{Claim 3:} $(K,v)$ is algebraically maximal.

\smallskip
\emph{Proof of Claim 3:} We show that all three places from our decomposition of $v$ 
are algebraically maximal.

By Claim 1, $K_2$ is an NIP field of characteristic $p$. If it is finite, then $K_2 = Kv$. Otherwise, by Claim 2 and \cite[Theorem 4.3]{KSW}, $K_2$ has no 
finite extensions of degree divisible by $p$. 
Thus, any henselian valuation on $K_2$ with residue characteristic $p$ is 
algebraically maximal and has a $p$-divisible value group
(cf.~\cite[Remark 6.7]{Joh15}).
Hence the place $K_2 \to Kv$ is algebraically maximal.

By saturation, any countable chain of balls of the place
$K \to Kv$ has non-empty intersection, thus the same holds for the place $K_1\to K_2$. As $\Delta/\Delta_0$ is archimedean, 
in fact any chain of balls with respect to this place has a non-empty intersection. Thus, the place $K_1 \to K_2$ is
spherically complete and therefore defectless (and so in particular algebraically maximal).

Finally, the place $K \to K_1$ is algebraically maximal as it is henselian of equicharacteristic $0$. Hence, we conclude that 
$v$ is algebraically maximal. This
proves the claim.
\smallskip

Next, we show that $(K,v)$ is either finitely ramified by $p$-divisible or roughly $p$-divisible.
Note first that $\Delta_0$ is $p$-divisible: if it is not trivial, then $K_2$ is a perfect infinite NIP field of $\mathrm{char}(K_2)=p$ and 
thus has no separable extensions of degree divisible by $p$ (and hence any
valuation on $K_2$ has $p$-divisible value group). 
Let $\Delta_p$ be the largest $p$-divisible convex subgroup of $\Gamma$. Then, $\Delta_p$ is a \emph{definable}
subgroup of $\Gamma$. 
Assume that $v$ is not finitely ramified by $p$-divisible, 
then (by saturation!) $\Delta_0$ cannot be definable in $\Gamma$.
Thus, we get $\Delta_0 \lneq \Delta_p$ and hence
$\Delta \leq \Delta_p$. In particular, $vK$ is roughly $p$-divisible.

Finally, we show that if $Kv$ is finite, then $v$ is finitely ramified. If $Kv$ is finite, then \cite[Proposition 5.3]{KSW} 
implies that $v_2$ is trivial. By saturation, $v$ must be finitely ramified.
\end{proof}

We can now state and prove our second main result. Recall for a given prime $p$,
 a \emph{$p$-adically
closed field} is a field which is elementarily
equivalent to $\mathbb{Q}_p$. A \emph{$\mathcal{P}$-adically closed field} is a
finite extension of a $p$-adically closed field. For more on 
$\mathcal{P}$-adically closed fields see \cite{PrRo}.
\label{sec4}
\begin{Thm} \label{main2}
Let $(K,v)$ be a henselian valued field of residue characteristic 
$\mathrm{char}(Kv)=p$. In case $(K,v)$ has mixed characteristic, 
assume that
$K^\times/(K^\times)^p$ is finite. 
If the characteristic of $K$ is positive, assume that $K$ has finite degree
of imperfection.
Then
$$(K,v) \textrm{ is NIP }\Longleftrightarrow Kv \textrm{ is NIP and }(K,v)\textrm{ is roughly separably tame}.$$ 
\end{Thm}
\begin{proof}
The implication from left to right follows from the fact that the residue field $Kv$
is interpretable in $(K,v)$ and what we have shown so far:
In mixed characteristic, this is the statement of Proposition \ref{rtame}; 
the positive characteristic case is treated in Corollary \ref{Cor4.1}. All 
henselian valued fields of equicharacteristic $0$ are tame.

For the converse, assume that $Kv$ is NIP and $(K,v)$ is roughly separably tame. 
In case $(K,v)$ has equicharacteristic, the result follows from Corollary
\ref{equi}. 
Thus, we may assume that we have 
$(\mathrm{char}(K), \mathrm{char}(Kv))=(0,p)$ for some $p>0$. Furthermore, we may assume
that $(K,v)$ is sufficiently saturated.

We again write $\Gamma:=vK$ and decompose $v$ as in the proof of Proposition \ref{rtame}:
let $\Delta_0 \leq \Gamma$ be the biggest convex subgroup not containing $v(p)$ and let $\Delta \leq \Gamma$ be
the smallest convex subgroup containing $v(p)$. We get the following decomposition of the place $\varphi_v:K \to Kv$
corresponding to $v$:
$$K =K_0 \xlongrightarrow{\Gamma/\Delta} K_1 \xlongrightarrow{\Delta/\Delta_0} K_2  \xlongrightarrow{\Delta_0} K_3=Kv,$$
where every arrow is labelled with the corresponding value group.
Note that $\mathrm{char}(K)=\mathrm{char}(K_1)=0$ and $\mathrm{char}(K_2)= \mathrm{char}(Kv)=p$.
Let $v_i$ denote the valuation on $K_i$ corresponding to the place $K_i \to K_{i+1}$.

By the first part of Lemma \ref{finite}, 
we have that $|K_1^\times/(K_1^\times)^p|$ is finite. Thus, applying the
second part of Lemma
\ref{finite}, we get that
$K_2$ is perfect. Moreover, $v_2$ has $p$-divisible value group, perfect residue field and is separably algebraically maximal.
Thus, by Corollary \ref{equi}, $(K_2,v_2)$ is NIP. 
In particular, $K_2$ is either finite or admits no Galois extensions of degree divisible by $p$. In case $K_2$ is finite, $K_3=K_2$ is obviously
stably embedded.
If $K_2$ admits no Galois extensions of degree divisible by $p$, $K_3$ is stably embedded in $(K_2,v_2)$ by Lemma 
\ref{SE}. 

In case $K_2 = Kv$ is finite, the definition of rough tameness implies that $vK$ is finitely ramified. In particular,
$(K_1,v_1)$ is a finitely ramified henselian valued field with finite residue field. Thus, by the axiomatization of $\mathcal{P}$-adically closed fields (see \cite[Theorem 3.1]{PrRo}), we get that 
$(K_1,v_1)$ is $\mathcal{P}$-adically closed. As any $p$-adically closed field
is NIP, we conclude that $(K_1,v_1)$ is NIP. As $K_2$ is finite, $K_2$ is
stably embedded in $(K_1,v_1)$.

Now assume that $K_2$ admits no Galois extensions of degree divisible by $p$.
Since $(K,v)$ is roughly separably tame, $(K_1,v_1)$ is either finitely ramified
or its value group $v_1K_1$ is $p$-divisible. If $v_1K_1$ is $p$-divisible, then
 $(K_1,v_1)$ is Kaplansky. As $(K,v)$ is algebraically maximal, we get that
$(K_1,v_1)$ is algebraically maximal. Thus, by Theorem 
\ref{amK}, $(K_1,v_1)$ is also NIP.
Moreover $K_2$ is stably embedded in $(K_1,v_1)$ by Lemma \ref{SE}.
On the other hand, if $(K_1,v_1)$ is finitely ramified then it is also NIP by 
\cite[Corollaire 7.5]{Bel}. In this case, $K_2$ is stably embedded in 
$(K_1,v_1)$ by \cite[Theorem 7.3]{vdD}.

In particular, $K_1$ is an NIP field of characteristic $0$, so $(K,v_0)$ is also NIP. Note that its residue field $K_1$ is stably embedded by Lemma \ref{SE}.

What we have shown so far is that we can decompose $v=v_2 \circ v_1 \circ v_0$
such that $(K_i,v_i)$ is NIP. Moreover, we have shown 
that the residue field $K_iv_i$ is
stably embedded in $(K_i,v_i)$. By Proposition \ref{se}, 
we conclude that $(K,v)$ is NIP.
\end{proof}

Finally, we remark that there are only few finitely ramified NIP henselian 
valued fields $K$ of mixed characteristic $(0,p)$ for which 
$K^\times/(K^\times)^p$ is finite: they are all elementarily equivalent to a generalized power
series fields over $\mathcal{P}$-adically closed fields.
\begin{Cor}
Let $(K,v)$ be a finitely ramified henselian valued NIP field of mixed characteristic $(0,p)$ with $K^\times/(K^\times)^p$ finite. 
Then $Kv$ is finite and there is some coarsening
$u$ of $v$ such that $Ku$ is $\mathcal{P}$-adically closed.
\end{Cor}
\begin{proof} Assume $(K,v)$ is a finitely ramified henselian valued NIP field of mixed characteristic $(0,p)$ with $K^\times/(K^\times)^p$ finite.
We write $\Gamma:=vK$. As in the proof of Proposition \ref{rtame} we decompose $v$.
Let again $\Delta \leq \Gamma$ be the smallest convex subgroup containing $v(p)$ (and note that by finite ramification, the biggest convex subgroup not containing $v(p)$ is trivial). Thus, we get:
$$K \xlongrightarrow{\Gamma/\Delta} K_1 \xlongrightarrow{\Delta} Kv,$$
where every arrow is labelled with the corresponding value group. Let $v_1$ denote the valuation on $K_1$
corresponding to the place $K_1 \to Kv$. Note that by finite ramification, the value group $v_1K_1$ is not $p$-divisible.
Then, we have $\mathcal{O}_{v_1}[\frac{1}{p}]=K_1$ and (by Lemma 
\ref{finite}) that
$|K_1^\times/(K_1^\times)^p|$ is finite. Now \cite[Proposition 3.2]{Koe04} implies that we have $v_1K_1 \cong \mathbb{Z}$
and $Kv$ finite. Thus, by the axiomatization of $\mathcal{P}$-adically closed fields (see \cite[Theorem 3.1]{PrRo}), we get that 
$K_1$ is $\mathcal{P}$-adically closed.
\end{proof}

\bibliographystyle{alpha}
\bibliography{franzi}
\end{document}